\documentclass[11pt,leqno]{amsart}
\usepackage{amsmath,amsthm,amssymb}
\usepackage{tabularx}
\usepackage{comment}
\usepackage{url} 
\usepackage{color}

\newcommand{\GF}{{\mathbb F}}
\newcommand{\FF}{{\mathbb F}}

\newcommand{\R}{{\mathbb R}}
\newcommand{\RR}{{\mathbb R}}

\newcommand{\Aut}{{\rm Aut}}

\newcommand{\wt}{{\rm wt}}
\newcommand{\supp}{{\rm supp}}

\DeclareMathOperator{\Harm}{Harm}

\newtheorem{thm}{Theorem}[section]
\newtheorem{lem}[thm]{Lemma}

\theoremstyle{definition}
\newtheorem{df}[thm]{Definition}
\newtheorem{rem}[thm]{Remark}
\newtheorem{ex}[thm]{Example}

\numberwithin{equation}{section}

\title[
Exceptional designs in some extended quadratic residue codes
]
{
Exceptional designs in some extended quadratic residue codes
}

\author{Reina Ishikawa*}
\thanks{*Corresponding author}
\address{		Graduate School of Science and Engineering, 
		Waseda University, 
		Tokyo 169--8555, Japan
}
\email{reina.i@suou.waseda.jp} 



\begin{document}
\begin{abstract}
In the present paper, 
we give proofs of the existence of a 3-design 
in the extended ternary quadratic residue code of length 14 
and the extended quaternary quadratic residue code of length 18. 
\end{abstract}
\maketitle


\section{Introduction}

Let $C$ be a code over $\FF_q$ and 
$C_\ell:=\{c\in C\mid \wt(c)=\ell\}$. 
If $C_\ell$ is non-empty, then 
we call $C_\ell$ a shell of the code $C$. 
For any shell, its extended quadratic residue code has a combinatorial $2$-design by the transitivity argument  (see Example \ref{ex:AutQR}). 
That is, if $C$ denotes the extended quadratic residue code and $\mathcal{B}(C_\ell):=\{\supp(x)\mid x\in C_\ell\}$, then $\mathcal{B}(C_\ell)$ forms the set of blocks of a combinatorial $2$-design. 
Herein, except for Remark \ref{rem:except}, we always assume that 
a combinatorial $t$-design allows the existence of 
repeated blocks.

Miezaki and Nakasora \cite{MN-TEC} introduced the following notation:
\begin{align*}
\delta(C)&:=\max\{t\in \mathbb{N}\mid \forall w, 
C_{w} \mbox{ is a } t\mbox{-design}\},\\ 
s(C)&:=\max\{t\in \mathbb{N}\mid \exists w \mbox{ s.t.~} 
C_{w} \mbox{ is a } t\mbox{-design}\}.
\end{align*}
We remark that $\delta(C) \leq s(C)$ holds, and 
we consider the possible occurrence of $\delta(C)<s(C)$.
Miezaki and Nakasora \cite{MN-TEC} gave the first examples 
that support combinatorial $t$-designs for all weights 
obtained from the Assmus--Mattson theorem and 
that support $t'$-designs for some weights with some $t'>t$ 
(see also \cite{{BS},{Dillion-Schatz},{MMN},{mn-typeI}}).

After that, 
an example was given by Bonnecaze and S\'ole \cite{BS}, who 
found a 3-design in 
the extended binary quadratic residue code of length 42. 
They showed the existence of this design by 
electronic calculation and 
noted that 
this design ``cannot be derived from the Assmus--Mattson theorem, 
and does not follow by the standard transitivity argument." 
The goal of the present paper is to give more examples in
extended ternary and quaternary quadratic residue codes. 

Herein, 
we compute Jacobi polynomials and harmonic weight enumerators of 
the extended ternary quadratic residue code of length 14
and the extended quaternary quadratic residue code of length 18, 
then we give an alternative approach 
to the existence of a 3-design in 
those codes. 

\begin{thm}\label{thm:main1}
Let $C$ be the extended ternary quadratic residue code of length $14$. 
Then we have the following. 
\begin{enumerate}
\item [(1)]
For any $\ell\neq 10$, 
$C_{\ell}$ is a combinatorial $2$-design and not a $3$-design. %
\item [(2)]
$C_{10}$ is a combinatorial $3$-design $(3$-$(14, 10, 180)$ design$)$ and not a $4$-design. 
\end{enumerate}
Hence, we have $2=\delta(C) < s(C)=3$. 
\end{thm}

\begin{thm}\label{thm:main2}
Let $C$ be the extended quaternary quadratic residue code of length $18$. 
Then we have the following. 
\begin{enumerate}
\item [(1)]
For any $\ell\neq 13$, 
$C_{\ell}$ is a combinatorial $2$-design and not a $3$-design. 
\item [(2)]
$C_{13}$ is a combinatorial $3$-design $(3$-$(18, 13, 18018)$ design$)$ and not a $4$-design. 
\end{enumerate}
Hence, we have $2=\delta(C) < s(C)=3$. 
\end{thm}

\begin{rem}\label{rem:except}
In this remark, 
we assume that 
a combinatorial $t$-design does not allow the 
existence of repeated blocks.

Let $C$ be the extended quaternary quadratic residue code of length 18. 
Then by {\sc Magma}, we have the following. 
\begin{enumerate}
\item [(1)]
For any $\ell\neq 10$, 
$C_{\ell}$ is a combinatorial $2$-design and not a $3$-design. 
\item [(2)]
$C_{10}$ is a combinatorial $3$-design (3-(18, 10, 315)) and not a $4$-design. 
\end{enumerate}
\end{rem}

We remark that some of the designs in 
Theorems \ref{thm:main1} and \ref{thm:main2} and 
Remark \ref{rem:except} 
cannot be derived from the Assmus--Mattson theorem 
and do not follow by the standard transitivity argument 
(see Remark \ref{rem:final}).



This paper is organized as follows. 
In Section~\ref{sec:pre}, 
we give definitions and some basic properties of the 
codes, 
combinatorial $t$-designs, 
Jacobi polynomials, and harmonic weight enumerators 
used herein.
In Sections~\ref{sec:mainA} and \ref{sec:mainB}, 
we give proofs of Theorems~\ref{thm:main1} and \ref{thm:main2}, 
respectively. 

All computer calculations presented in this paper were done using 
{\sc Magma}~\cite{Magma} and {\sc Mathematica}~\cite{Mathematica}. 

\section{Preliminaries}\label{sec:pre}

In this section, 
we give definitions and some basic properties of the 
codes, 
combinatorial $t$-designs, 
Jacobi polynomials, and harmonic weight enumerators 
used in this paper. 
Most parts are quoted from \cite{{AMMN},{Bachoc}}.

\subsection{Codes and combinatorial $t$-designs}

A linear code $C$ of length $n$ is a linear subspace of $\FF_{q}^{n}$. 
An inner product $({x},{y})$ on $\FF_q^n$ is given 
by
\[
(x,y)=\sum_{i=1}^nx_iy_i,
\]
where $x,y\in \FF_q^n$ with $x=(x_1,x_2,\ldots, x_n)$ and 
$y=(y_1,y_2,\ldots, y_n)$. 
The dual of a linear code $C$ is defined as follows: 
\[
C^{\perp}=\{{y}\in \FF_{q}^{n}\mid ({x},{y}) =0 \text{ for all }{x}\in C\}. 
\]
For $x \in\FF_q^n$,
the weight $\wt(x)$ is the number of its nonzero components. 

Let $C$ be a cyclic code of length $n$; 
i.e., if $(c_1,\ldots,c_n)\in C$ then 
$(c_2,\ldots,c_1)\in C$. 
Then 
$C$ corresponds to an ideal $(f)$ of 
\[
\FF_q[x]/(x^n-1). 
\]
We call $f$ a generator polynomial of $C$. 
For odd prime $p$ such that 
$q$ is a quadratic residue modulo $p$, 
the quadratic residue code is 
a cyclic code of length $p$, 
which is generated by 
\[
\prod_{\ell \in (\FF_p^\ast)^2}(x-\alpha^\ell), 
\]
where $\alpha$ is a primitive $p$-th root of unity. 
For the details of the quadratic residue codes, see \cite{{MS},{HP}}. 

Let $C$ be a code of length $n$. 
Then the symmetric group $S_n$ acts on the $n$ coordinates of $C$. 
The automorphism group $\Aut(C)$ of $C$ is the subgroup of $S_n$ such that 
\[
\Aut(C):=\{\sigma\in S_n\mid C^\sigma=C\}, 
\]
where 
\[
C^\sigma:=\{c=(c_1,\ldots,c_n)\in C \mid c=(c_{\sigma(1)},\ldots,c_{\sigma(n)})\}. 
\]
\begin{ex}\label{ex:AutQR}
The automorphism group of the extended ternary or quaternary quadratic residue code of length $p+1$ is $PSL_2(p)$ except for the three cases in \cite[Chapter 6, Theorem 6.6.27]{HP} (see also \cite{{assmus-mattson},{Blahut},{Huffman}}). However, note that these three cases are not among those treated in 
the present paper. 

We assume that the coordinates of 
$\widetilde{Q}_{p+1}$ 
are labeled by $\{\infty,0,1,\ldots,p-1\}$ and 
identify $\{\infty,0,1,\ldots,p-1\}$ with $PG(1,p)$. 
Let $p\equiv 1\pmod{4}$. 
Then the action of $PGL_2(p)$ on $PG(1,p)$ is $3$-transitive 
(see \cite[Propositions 4.6 and 4.8]{BJL}) and 
the action of $PSL_2(p)$ on $PG(1,p)$ 
is not 3-homogeneous 
(see \cite{BR}). 
\end{ex}

A combinatorial $t$-design 
is a pair 
$\mathcal{D}=(\Omega,\mathcal{B})$, where $\Omega$ is a set of points of 
cardinality $v$, and $\mathcal{B}$ is a collection of $k$-element subsets
of $\Omega$ called blocks, with the property that any $t$ points are 
contained in precisely $\lambda$ blocks.

The support of a vector 
${x}:=(x_{1}, \dots, x_{n})$, 
$x_{i} \in \GF_{q}$ is 
the set of indices of its nonzero coordinates: 
${\rm supp} ({x}) = \{ i \mid x_{i} \neq 0 \}$\index{$supp (x)$}.
Let $\Omega:=\{1,\ldots,n\}$ and 
$\mathcal{B}(C_\ell):=\{\supp({x})\mid {x}\in C_\ell\}$. 
Then for a code $C$ of length $n$, 
we say that $C_\ell$ is a combinatorial $t$-design if 
$(\Omega,\mathcal{B}(C_\ell))$ is a combinatorial $t$-design \cite{{BMS},{DT}}.

The following theorem from Assmus and Mattson \cite{assmus-mattson} is one of the 
most important theorems in coding theory and design theory.
\begin{thm}[\cite{assmus-mattson}] \label{thm:assmus-mattson}

Let $C$ be a linear code of 
length $n$ over $\FF_q$ with minimum weight $d$. 
Let $C^\perp$ denote the dual code of $C$, with 
minimum weight $d^\perp$. 
Suppose that an integer $t$ $($$1 \leq t \leq n$$)$ is 
such that there are at most $d-t$ weights of $C^\perp$ 
in $\{1, 2,\ldots , n - t\}$, 
or such that there are at most $d^\perp - t$ weights of $C$ 
in $\{1, 2, \ldots , n-t\}$. 
Then the supports of the words of any fixed weight 
in $C$ form a $t$-design $($with possibly repeated blocks$)$.
\end{thm}

The following lemma is easily seen. 

\begin{lem}[{{\cite[Page 3, Proposition 1.4]{CL}}}]\label{lem: divisible}
Let $\lambda(S)$ be the number of blocks containing a given set $S$
of $s$ points in a combinatorial $t$-$(v,k,\lambda)$ design, where $0\leq s\leq t$. Then
\[
\lambda(S)\binom{k-s}{t-s}
=
\lambda\binom{v-s}{t-s}. 
\]
\end{lem}

\subsection{Jacobi polynomials}

Let $C$ be a binary code of length $n$ and $T\subset [n]:=\{1,\ldots,n\}$. 
Then the Jacobi polynomial of $C$ with $T$ is defined as follows \cite{Ozeki}:
\[
J_{C,T} (w,z,x,y) :=\sum_{c\in C}w^{m_0(c)} z^{m_1(c)}x^{n_0(c)}y^{n_1(c)}, 
\]
where for $c=(c_1,\ldots,c_n)$, 
\begin{align*}
m_i(c)&=|\{j\in T\mid c_j=i \}|,\\
n_i(c)&=|\{j\in [n]\setminus T\mid c_j=i \}|.
\end{align*}
\begin{rem}\label{rem:hom}

It is easy to see that 
$C_\ell$ is a combinatorial $t$-design $($with possibly repeated blocks$)$
if and only if 
the coefficient of $z^{t}x^{n-\ell}y^{\ell-t}$ 
in $J_{C,T}$ is independent of the choice of $T$ with $|T|=t$. 

\end{rem}

\subsection{Harmonic weight enumerators}\label{sec:Har}


In this subsection, we review the concept of 
harmonic weight enumerators.

Let $\Omega=\{1, 2,\ldots,n\}$ be a finite set (which will be the set of coordinates of the code) and 
let $X$ be the set of its subsets, while, for all $k= 0,1,\dots, n$, 
$X_{k}$ is the set of its $k$-subsets.
We denote by $\R X$ and $\R X_k$ the 
real vector spaces spanned by the elements of $X$
and $X_{k}$, respectively.
An element of $\R X_k$ is denoted by
$$f=\sum_{z\in X_k}f(z)z$$
and is identified with the real-valued function on $X_{k}$ given by 
$z \mapsto f(z)$. 

An element $f\in \R X_k$ can be extended to an element $\widetilde{f}\in \R X$ by setting, for all $u \in X$,
$$\widetilde{f}(u)=\sum_{z\in X_k, z\subset u}f(z).$$
If an element $g \in \R X$ is equal to $\widetilde{f}$ 
for some $f \in \R X_{k}$, then we say that $g$ has degree $k$. 
The differentiation $\gamma$ is the operator on $\RR X$ defined by linearity from 
$$\gamma(z) =\sum_{y\in X_{k-1},y\subset z}y$$
for all $z\in X_k$ and for all $k=0,1, \ldots n$, and $\Harm_{k}$ is the kernel of $\gamma$:
$$\Harm_k =\ker(\gamma|_{\R X_k}).$$

\begin{thm}[{{\cite[Theorem 7]{Delsarte}}}]\label{thm:design}
A set $\mathcal{B} \subset X_{m}$ $($where $m \leq n$$)$ of blocks is a $t$-design 
if and only if $\sum_{b\in \mathcal{B}}\widetilde{f}(b)=0$ 
for all $f\in \Harm_k$, $1\leq k\leq t$. 
\end{thm}
Then the symmetric group $S_n$ acts on $\Omega$. 
The automorphism group $\Aut(B)$ of $B$ is 
the subgroup of $S_n$ such that 
\[
\Aut(B):=\{\sigma\in S_n\mid B^\sigma=B\}. 
\]
Then $\Aut(B)$ acts on $\Harm_k$ through the above action, and 
we denote by $\Harm_k^{\Aut(B)}$ the set of the invariants of $\Aut(B)$. 
Then we have the following. 
\begin{thm}[\cite{AMMN}]\label{thm:design-aut}
A set $\mathcal{B} \subset X_{m}$ $($where $m \leq n$$)$ of blocks is a $t$-design 
if and only if $\sum_{b\in \mathcal{B}}\widetilde{f}(b)=0$ 
for all $f\in \Harm_k^{\Aut(B)}$, $1\leq k\leq t$. 
\end{thm}

\begin{proof}
This theorem was proved by Awada et~al.~\cite{AMMN}, and for convenience we quote their proof below. 

We assume that 
$\mathcal{B} \subset X_{m}$ is a $t$-design. 
Let $G=\Aut(B)$ and 
$f\in \Harm_k^{G}$ ($1\leq k\leq t$). 
Then by Theorem \ref{thm:design}, 
$\sum_{b\in \mathcal{B}}\widetilde{f}(b)=0$. 

We assume that 
for all $f\in \Harm_k^{G}$ ($1\leq k\leq t$), 
$\sum_{b\in \mathcal{B}}\widetilde{f}(b)=0$. 
Let 
\[
B=Gx_1\sqcup \cdots \sqcup G{x_\ell}, 
\]
and $f\in \Harm_k$ ($1\leq k\leq t$). 
Then 
\begin{align*}
\sum_{b\in \mathcal{B}}\widetilde{f}(b)
&=
\sum_{x\in Gx_1}\widetilde{f}(x)+ \cdots + \sum_{x\in G{x_\ell}}\widetilde{f}(x)\\
&=
\frac{1}{|G_{x_1}|}\sum_{g\in G}g\widetilde{f}(x_1)+ \cdots 
+ \frac{1}{|G_{x_\ell}|}\sum_{g\in G}g\widetilde{f}(x_\ell)=0, 
\end{align*}
since for each $i\in \{1,\ldots,\ell\}$, 
\[
\frac{1}{|G_{x_i}|}\sum_{g\in G}g\widetilde{f}
\]
is an invariant polynomial. 
\end{proof}

In \cite{Bachoc}, the harmonic weight enumerator associated with a binary linear code $C$ was defined as follows. 
\begin{df}
Let $C$ be a binary code of length $n$ and let $f\in\Harm_{k}$. 
The harmonic weight enumerator associated with $C$ and $f$ is

$$w_{C,f}(x,y)=\sum_{{c}\in C}\widetilde{f}({c})x^{n-\wt({c})}y^{\wt({c})}.$$
\end{df}


It follows from Theorem \ref{thm:design} that 
$C_\ell$ is a combinatorial $t$-design $($with possibly repeated blocks$)$ 
if and only if 
the coefficient of $x^{n-\ell}y^\ell$ in $w_{C,f}(x,y)$ vanishes.

\section{Proofs of Theorem \ref{thm:main1}}\label{sec:mainA}
Let $C$ be the extended ternary quadratic residue code of length 14. 
Before giving proofs of Theorem \ref{thm:main1}, 
we give Jacobi polynomials and harmonic weight enumerators of $C$. 

\subsection{Jacobi polynomials of the extended ternary quadratic residue code of length 14}

In this subsection, 
we give Jacobi polynomials of $C$.

\begin{thm}\label{thm:Jacobi1}
Let $C$ be the extended ternary quadratic residue code of length $14$,
$X:=\{1,2,\ldots,14\}$, 
$T\in \binom{X}{3}$, and $G=\Aut(C)$. 
Then 
\begin{align*}
\binom{X}{3}&=G\{4, 9, 11\}\sqcup G\{1, 4, 5\}\sqcup G\{1, 4, 11\}\sqcup G\{6, 9, 14\}\\
			&\sqcup G\{1, 7, 8\}\sqcup G\{2, 7, 14 \}\sqcup G\{ 1, 11, 13 \}\sqcup G\{ 1, 6, 13  \},
\end{align*}
and we have the following. 
\begin{enumerate}
\item [(1)]
If $T\in G\{4,9,11\}$, then
\begin{align*}
J&_{C,T}(w,z,x,y)=w^{3} x^{11} + 30 w^{3} x^{5} y^{6} + 12 w^{3} x^{4} y^{7} + 18 w^{3} x^{3} y^{8} + 14 w^{3} x^{2} y^{9} \\
&+ 6 w^{3} x y^{10} + 78 w^{2} x^{6} y^{5} z + 72 w^{2} x^{5} y^{6} z + 126 w^{2} x^{4} y^{7} z + 78 w^{2} x^{3} y^{8} z \\
&+ 90 w^{2} x^{2} y^{9} z + 36 w^{2} x y^{10} z + 6 w^{2} y^{11} z + 66 w x^{7} y^{4} z^{2} + 54 w x^{6} y^{5} z^{2} \\
&+ 162 w x^{5} y^{6} z^{2} + 192 w x^{4} y^{7} z^{2} + 270 w x^{3} y^{8} z^{2} + 162 w x^{2} y^{9} z^{2} \\
&+ 66 w x y^{10} z^{2} + 8 x^{8} y^{3} z^{3} + 18 x^{7} y^{4} z^{3} + 58 x^{6} y^{5} z^{3} + 80 x^{5} y^{6} z^{3} \\
&+ 180 x^{4} y^{7} z^{3} + 166 x^{3} y^{8} z^{3} + 110 x^{2} y^{9} z^{3} + 28 y^{11} z^{3}.
\end{align*}

\item [(2)]
If $T\in G\{1, 4, 5\}$, then
\begin{align*}
J_{C,T}&(w,z,x,y)=w^3 x^{11} + 26 w^3 x^5 y^6 + 18 w^3 x^4 y^7 + 22 w^3 x^3 y^8 + 6 w^3 x^2 y^9 \\
&+ 6 w^3 x y^{10} + 2 w^3 y^{11} + 90 w^2 x^6 y^5 z + 54 w^2 x^5 y^6 z + 114 w^2 x^4 y^7 z \\
&+ 102 w^2 x^3 y^8 z + 90 w^2 x^2 y^9 z + 30 w^2 x y^{10} z + 6 w^2 y^{11} z + 54 w x^7 y^4 z^2 \\
&+ 72 w x^6 y^5 z^2 + 174 w x^5 y^6 z^2 + 168 w x^4 y^7 z^2 + 270 w x^3 y^8 z^2 + 168 w x^2 y^9 z^2 \\
&+ 66 w x y^{10} z^2 + 12 x^8 y^3 z^3 + 12 x^7 y^4 z^3 + 54 x^6 y^5 z^3 + 88 x^5 y^6 z^3 \\
&+ 180 x^{4} y^{7} z^{3} + 164 x^{3} y^{8} z^{3} + 110 x^{2} y^{9} z^{3} + 28 y^{11} z^{3}.
\end{align*}

\item [(3)]
If $T\in G\{1, 4, 11\}$, then
\begin{align*}
J&_{C,T}(w,z,x,y)=w^3 x^{11} + 30 w^3 x^5 y^6 + 12 w^3 x^4 y^7 + 18 w^3 x^3 y^8 + 14 w^3 x^2 y^9 \\
&+ 6 w^3 x y^{10} + 78 w^2 x^6 y^5 z + 72 w^2 x^5 y^6 z + 126 w^2 x^4 y^7 z + 78 w^2 x^3 y^8 z \\
&+ 90 w^2 x^2 y^9 z + 36 w^2 x y^{10} z + 6 w^2 y^{11} z + 66 w x^7 y^4 z^2 + 54 w x^6 y^5 z^2 \\
&+ 162 w x^5 y^6 z^2 + 192 w x^4 y^7 z^2 + 270 w x^3 y^8 z^2 + 162 w x^2 y^9 z^2 \\
&+ 66 w x y^10 z^2 + 8 x^8 y^3 z^3 + 18 x^7 y^4 z^3 + 58 x^6 y^5 z^3 + 80 x^5 y^6 z^3 \\
&+ 180 x^4 y^7 z^3 + 166 x^3 y^8 z^3 + 110 x^2 y^9 z^3 + 28 y^{11} z^3.
\end{align*}

\item [(4)]
If $T\in G\{6, 9, 14\}$, then
\begin{align*}
J_{C,T}&(w,z,x,y)=w^3 x^{11} + 26 w^3 x^5 y^6 + 18 w^3 x^4 y^7 + 22 w^3 x^3 y^8 + 6 w^3 x^2 y^9 \\
&+ 6 w^3 x y^{10} + 2 w^3 y^{11} + 90 w^2 x^6 y^5 z + 54 w^2 x^5 y^6 z + 114 w^2 x^4 y^7 z \\
&+ 102 w^2 x^3 y^8 z + 90 w^2 x^2 y^9 z + 30 w^2 x y^{10} z + 6 w^2 y^{11} z + 54 w x^7 y^4 z^2 \\
&+ 72 w x^6 y^5 z^2 + 174 w x^5 y^6 z^2 + 168 w x^4 y^7 z^2 + 270 w x^3 y^8 z^2 + 168 w x^2 y^9 z^2 \\
&+ 66 w x y^{10} z^2 + 12 x^8 y^3 z^3 + 12 x^7 y^4 z^3 + 54 x^6 y^5 z^3 + 88 x^5 y^6 z^3 \\
&+ 180 x^{4} y^{7} z^{3} + 164 x^{3} y^{8} z^{3} + 110 x^{2} y^{9} z^{3} + 28 y^{11} z^{3}.
\end{align*}

\item [(5)]
If $T\in G\{1, 7, 8\}$, then
\begin{align*}
J_{C,T}&(w,z,x,y)=w^3 x^{11} + 26 w^3 x^5 y^6 + 18 w^3 x^4 y^7 + 22 w^3 x^3 y^8 + 6 w^3 x^2 y^9 \\
&+ 6 w^3 x y^{10} + 2 w^3 y^{11} + 90 w^2 x^6 y^5 z + 54 w^2 x^5 y^6 z + 114 w^2 x^4 y^7 z \\
&+ 102 w^2 x^3 y^8 z + 90 w^2 x^2 y^9 z + 30 w^2 x y^{10} z + 6 w^2 y^{11} z + 54 w x^7 y^4 z^2 \\
&+ 72 w x^6 y^5 z^2 + 174 w x^5 y^6 z^2 + 168 w x^4 y^7 z^2 + 270 w x^3 y^8 z^2 + 168 w x^2 y^9 z^2 \\
&+ 66 w x y^{10} z^2 + 12 x^8 y^3 z^3 + 12 x^7 y^4 z^3 + 54 x^6 y^5 z^3 + 88 x^5 y^6 z^3 \\
&+ 180 x^{4} y^{7} z^{3} + 164 x^{3} y^{8} z^{3} + 110 x^{2} y^{9} z^{3} + 28 y^{11} z^{3}.
\end{align*}

\item [(6)]
If $T\in G\{ 2, 7, 14 \}$, then
\begin{align*}
J_{C,T}&(w,z,x,y)=w^3x^{11} + 30w^3x^5y^6 + 12w^3x^4y^7 + 18w^3x^3y^8 + 14w^3x^2y^9 \\
	&+ 6w^3xy^{10} + 78w^2x^6y^5z + 72w^2x^5y^6z + 126w^2x^4y^7z + 78w^2x^3y^8z \\
	&+ 90w^2x^2y^9z + 36w^2xy^{10}z + 6w^2y^{11}z + 66wx^7y^4z^2 + 54wx^6y^5z^2 \\
	&+ 162wx^5y^6z^2 + 192wx^4y^7z^2 + 270wx^3y^8z^2 + 162wx^2y^9z^2 \\
	&+ 66wxy^{10}z^2 + 8x^8y^3z^3 + 18x^7y^4z^3 + 58x^6y^5z^3 + 80x^5y^6z^3 \\
	&+ 180x^4y^7z^3 + 166x^3y^8z^3 + 110x^2y^9z^3 + 28y^{11}z^3.\\
\end{align*}

\item [(7)]
If $T\in G\{ 1, 11, 13  \}$, then
\begin{align*}
J_{C,T}&(w,z,x,y)=w^3x^{11} + 30w^3x^5y^6 + 12w^3x^4y^7 + 18w^3x^3y^8 + 14w^3x^2y^9 \\
	&+ 6w^3xy^{10} + 78w^2x^6y^5z + 72w^2x^5y^6z + 126w^2x^4y^7z + 78w^2x^3y^8z \\
	&+ 90w^2x^2y^9z + 36w^2xy^{10}z + 6w^2y^{11}z + 66wx^7y^4z^2 + 54wx^6y^5z^2 \\
	&+ 162wx^5y^6z^2 + 192wx^4y^7z^2 + 270wx^3y^8z^2 + 162wx^2y^9z^2 \\
	&+ 66wxy^{10}z^2 + 8x^8y^3z^3 + 18x^7y^4z^3 + 58x^6y^5z^3 + 80x^5y^6z^3 \\
	&+180x^4y^7z^3 + 166x^3y^8z^3 + 110x^2y^9z^3 + 28y^{11}z^3.\\
\end{align*}
\item [(8)]
If $T\in G\{ 1, 6, 13 \}$, then
\begin{align*}
J_{C,T}&(w,z,x,y)=w^3 x^{11} + 26 w^3 x^5 y^6 + 18 w^3 x^4 y^7 + 22 w^3 x^3 y^8 + 6 w^3 x^2 y^9 \\
&+ 6 w^3 x y^{10} + 2 w^3 y^{11} + 90 w^2 x^6 y^5 z + 54 w^2 x^5 y^6 z + 114 w^2 x^4 y^7 z \\
&+ 102 w^2 x^3 y^8 z + 90 w^2 x^2 y^9 z + 30 w^2 x y^{10} z + 6 w^2 y^{11} z + 54 w x^7 y^4 z^2 \\
&+ 72 w x^6 y^5 z^2 + 174 w x^5 y^6 z^2 + 168 w x^4 y^7 z^2 + 270 w x^3 y^8 z^2 + 168 w x^2 y^9 z^2 \\
&+ 66 w x y^{10} z^2 + 12 x^8 y^3 z^3 + 12 x^7 y^4 z^3 + 54 x^6 y^5 z^3 + 88 x^5 y^6 z^3 \\
&+ 180 x^{4} y^{7} z^{3} + 164 x^{3} y^{8} z^{3} + 110 x^{2} y^{9} z^{3} + 28 y^{11} z^{3}.
\end{align*}

\end{enumerate}
\end{thm}
\begin{proof}
Let $G:=\Aut(C)$. Then 
$G$ acts on $\widetilde{X}:=\binom{X}{3}$. 
By {\sc Magma}, we have 
\begin{align*}
G\backslash \widetilde{X}=\{&\{4, 9, 11\},\{1, 4, 5\},\{1, 4, 11\},\{6, 9, 14\},
\\
&\{1, 7, 8\},\{2, 7, 14 \},\{ 1, 11, 13 \},\{ 1, 6, 13  \}\}. 
\end{align*}
Let 

 $T_1:=\{4, 9, 11 \}$,
 $T_2:=\{1, 4, 5\}$,
 $T_3:=\{1, 4, 11 \}$,
 $T_4:=\{6, 9, 14 \}$,

 $T_5:=\{1, 7, 8 \}$,
 $T_6:=\{2, 7, 14  \}$,
 $T_7:=\{1, 11, 13 \}$, and
 $T_8:=\{1, 6, 13 \}$.

It is sufficient to determine $J_{C,T_i}\ (i\in \{1,\ldots,8\})$, 
and we perform brute force enumeration based on the definition 
by using {\sc Magma}. 

\end{proof}
\subsection{Harmonic weight enumerators of the extended ternary quadratic residue code of length 14}
In this subsection, 
we give harmonic weight enumerators of $C$. 

\begin{thm}\label{thm:harm1}
Let $C$ be the extended ternary quadratic residue code of length $14$ and 
$f$ be a harmonic function of degree $3$, which is 
an invariant of $\Aut(C)$. 
Then there exists $c_f$ such that 
\begin{align*}
w_{C,f}(x,y)=&c_f\times(-312x^{8}y^{6} + 468x^{7}y^{7} + 312x^{6}y^{8} - 624x^{5}y^{9} + 156x^{3}y^{11}
). 
\end{align*}
\end{thm}

\begin{proof}
Let $G:=\Aut(C)$. By {\sc Magma}, 
$\Harm_3^{G}$ is a five-dimensional space, and 
let $\Harm_3^{G}=\langle f_1,\ldots, f_5\rangle$. 
The process of calculation is as follows. 
Firstly, we compute the $G$-orbit of $\binom{X}{3}$: 
\[
\binom{X}{3}=GT_1 \sqcup \cdots \sqcup G_{T_8}. 
\]
For $T=\{i_1,\ldots,i_t\}\in \binom{X}{t}$, 
$x_T:=x_{i_1}\cdots x_{i_t}$. 
For $i\in \{1,\ldots, 8\}$, 
let 
\[
g_i:=
\frac{1}{|G|}\sum_{g\in G}gx_{T_i}. 
\]
Secondly, we compute $a_1, ..., a_8$, which satisfy 
\[
\gamma(a_1g_1 + \cdots + a_8g_8) = 0.
\]

Note that $f_i\ (1\leq i\leq 5)$ are listed online by Ishikawa \cite{Ishikawa}. 
By {\sc Magma}, we obtain the results. 
\end{proof}

\begin{proof}[Proofs of Theorem \ref{thm:main1}]
Firstly, we show that $C_{10}$ is a $3$-design. 
We give two proofs. 
\begin{enumerate}
\item [(1)]
By Theorem \ref{thm:Jacobi1}, 
the coefficients of $z^{3} x^{14-10} y^{10-3}$ in 
$J_{C,T_i}$ $(i\in\{1,\ldots,8\})$ are the same. 
Hence, $C_{10}$ is a $3$-design. 
\item [(2)]
By Theorem \ref{thm:harm1}, 
the coefficient of $x^{4} y^{10}$ in 
$w_{C,f}$ is zero. 
Because of Theorem \ref{thm:design-aut}, $C_{10}$ is a $3$-design. 
\end{enumerate}

Secondly, we show that for $\ell\neq 10$, 
$C_{\ell}$ is not a $3$-design 
whenever $C_{\ell}$ is non-empty. 
We give two proofs.
\begin{enumerate}
\item[(1)]
By Theorem \ref{thm:Jacobi1},
for $\ell\neq 10$,
the coefficients of $z^{3} x^{14-\ell} y^{\ell-3}$ in
$J_{C,T_i}$ $(i\in\{1,\ldots,8\})$ are not the same.
Hence, $C_{\ell}$ is not a $3$-design.
\item[(2)]
By Theorem \ref{thm:harm1}, the coefficient of $x^{14-\ell} y^{\ell}$ in $w_{C,f}$ is non-zero. 
Because of Theorem \ref{thm:design-aut}, 
$C_{\ell}$ is not a $3$-design.	
\end{enumerate}

Thirdly, we show that $C_{10}$ is not a $4$-design. 
We assume the contrary. 
By Lemma \ref{lem: divisible}, 
\[
|C_{10}|\binom{10}{4}=\lambda\binom{14}{4}. 
\]
Then by {\sc Magma}, $|C_{10}|=546$, and 
we have 
\[
\lambda=\frac{1260}{11}. 
\]
This is a contradiction. 
\end{proof}


\section{Proofs of Theorem \ref{thm:main2}}\label{sec:mainB}

Let $C$ be the extended quaternary quadratic residue code of length 18. 
Before giving proofs of Theorem \ref{thm:main1}, 
we give Jacobi polynomials and harmonic weight enumerators of $C$.

\subsection{Jacobi polynomials of the extended quaternary quadratic residue code of length 18}
In this subsection, 
we give Jacobi polynomials of $C$. 

\begin{thm}\label{thm:Jacobi2}
Let $C$ be the extended quaternary quadratic residue code of length $18$, 
$X:=\{1,2,\ldots,18\}$, 
$T\in \binom{X}{3}$, and $G=\Aut(C)$. 
Then 
\begin{align*}
\binom{X}{3}&=G\{1, 3, 12\}\sqcup G\{6, 13, 17 \},
\end{align*}
and we have the following. 
\begin{enumerate}
\item [(1)]
If $T\in G\{1, 3, 12\}$, then
\begin{align*}
J&_{C,T}(w,z,x,y)=w^3x^{15} + 81w^3x^9y^6 + 72w^3x^7y^8 + 318w^3x^6y^9 + 855w^3x^5y^{10} \\
	&+ 1008w^3x^4y^{11} + 783w^3x^3y^{12} + 630w^3x^2y^{13} + 288w^3xy^{14} + 60w^3y^{15} \\
	&+ 153w^2x^{10}y^5z + 189w^2x^8y^7z + 1350w^2x^7y^8z + 4239w^2x^6y^9z \\
	&+ 6048w^2x^5y^{10}z + 7821w^2x^4y^{11}z + 8190w^2x^3y^{12}z + 5832w^2x^2y^{13}z \\
	&+ 2556w^2xy^{14}z + 486w^2y^{15}z + 63wx^{11}y^4z^2 + 171wx^9y^6z^2 \\
	&+ 1242wx^8y^7z^2 + 5481wx^7y^8z^2 + 10584wx^6y^9z^2 + 16587wx^5y^{10}z^2 \\
	&+ 24570wx^4y^{11}z^2 + 25416wx^3y^{12}z^2 + 17964wx^2y^{13}z^2 + 7290wxy^{14}z^2 \\
	&+ 1224wy^{15}z^2 + 9x^{12}y^3z^3 + 27x^{10}y^5z^3 + 354x^9y^6z^3 + 1818x^8y^7z^3 \\
	&+ 4392x^7y^8z^3 + 9387x^6y^9z^3 + 18018x^5y^{10}z^3 + 25380x^4y^{11}z^3 \\
	&+ 25932x^3y^{12}z^3 + 17010x^2y^{13}z^3 + 6120xy^{14}z^3 + 2145y^{15}z^3. \\
\end{align*}

\item [(2)]
If $T\in G\{6, 13, 17\}$, then
\begin{align*}
J_{C,T}&(w,z,x,y)=w^3x^{15} + 84w^3x^9y^6 + 63w^3x^7y^8 + 354w^3x^6y^9 + 846w^3x^5y^{10}\\
	&+ 882w^3x^4y^{11} + 912w^3x^3y^{12} + 630w^3x^2y^{13} + 270w^3xy^{14} + 54w^3y^{15} \\
	&+ 144w^2x^{10}y^5z + 216w^2x^8y^7z + 1242w^2x^7y^8z + 4266w^2x^6y^9z \\
	&+ 6426w^2x^5y^{10}z + 7434w^2x^4y^{11}z + 8190w^2x^3y^{12}z + 5886w^2x^2y^{13}z \\
	&+ 2574w^2xy^{14}z + 486w^2y^{15}z + 72wx^{11}y^4z^2 + 144wx^9y^6z^2 \\
	&+ 1350wx^8y^7z^2 + 5454wx^7y^8z^2 + 10206wx^6y^9z^2 + 16974wx^5y^{10}z^2 \\
	&+ 24570wx^4y^{11}z^2 + 25362wx^3y^{12}z^2 + 17946wx^2y^{13}z^2 + 7290wxy^{14}z^2 \\
	&+ 1224wy^{15}z^2 + 6x^{12}y^3z^3 + 36x^{10}y^5z^3 + 318x^9y^6z^3 + 1827x^8y^7z^3 \\
	&+ 4518x^7y^8z^3 + 9258x^6y^9z^3 + 18018x^5y^{10}z^3 + 25398x^4y^{11}z^3 \\
	&+ 25938x^3y^{12}z^3 + 17010x^2y^{13}z^3 + 6120xy^{14}z^3 + 2145y^{15}z^3.\\
\end{align*}

\end{enumerate}
\end{thm}
\begin{proof}
Let $G:=\Aut(C)$. Then 
$G$ acts on $\widetilde{X}:=\binom{X}{3}$. 
By {\sc Magma}, we have 
\[
G\backslash \widetilde{X}=\{\{1,3,12 \},\{6,13,17 \}\}. 
\]
Let $T_1:=\{1,3,12 \}$ and $T_2:=\{6,13,17 \}$. 

It is sufficient to determine $J_{C,T_i}\ (i\in \{1,2\})$, 
and we perform brute force enumeration based on the definition 
by using {\sc Magma}.

\end{proof}


\begin{thm}\label{thm:Jacobi3}
Let $C$ be the extended quaternary quadratic residue code of length $18$, 
$X:=\{1,2,\ldots,18\}$, 
$T\in \binom{X}{4}$, and $G=\Aut(C)$. 
Then 
\begin{align*}
\binom{X}{4}&=G\{ 9, 12, 14, 17\}\sqcup G\{ 4, 7, 14, 18 \}\sqcup G\{  2, 3, 5, 10 \}\sqcup G\{3, 6, 7, 17 \},
\end{align*}
and we have the following. 
\begin{enumerate}
\item [(1)]
If $T\in G\{ 9, 12, 14, 17\}$, then
\begin{align*}
J&_{C,T}(w,z,x,y)=w^4x^{14} + 54w^4x^8y^6 + 21w^4x^6y^8 + 168w^4x^5y^9 + 288w^4x^4y^{10} \\
&+ 144w^4x^3y^{11} + 258w^4x^2y^{12} +72w^4xy^{13} + 18w^4y^{14} + 120w^3x^9y^5z\\
& + 168w^3x^7y^7z + 744w^3x^6y^8z + 2232w^3x^5y^9z
 +  2952w^3x^4y^{10}z \\
&+ 2616w^3x^3y^{11}z + 2232w^3x^2y^{12}z + 1008w^3xy^{13}z + 216w^3y^{14}z \\
&+
    108w^2x^{10}y^4z^2 + 180w^2x^8y^6z^2 + 1368w^2x^7y^7z^2 + 5184w^2x^6y^8z^2\\& + 8424w^2x^5y^9z^2 +
    10944w^2x^4y^{10}z^2 + 13032w^2x^3y^{11}z^2 \\
&+ 10260w^2x^2y^{12}z^2 + 4824w^2xy^{13}z^2 + 972w^2y^{14}z^2 +
    24wx^{11}y^3z^3 \\
&+ 72wx^9y^5z^3 + 888wx^8y^6z^3 + 3816wx^7y^7z^3 + 7992wx^6y^8z^3 \\
&+
    15336wx^5y^9z^3 + 24072wx^4y^{10}z^3 + 26976wx^3y^{11}z^3 + 20712wx^2y^{12}z^3 \\
&+ 9072wxy^{13}z^3 + 
    1632wy^{14}z^3 + 18x^{10}y^4z^4 + 96x^9y^5z^4 + 873x^8y^6z^4 \\
&+ 2520x^7y^7z^4 + 5424x^6y^8z^4 +
    12000x^5y^9z^4 + 18654x^4y^{10}z^4 \\
&+ 20760x^3y^{11}z^4 + 14742x^2y^{12}z^4 + 5712xy^{13}z^4 + 2145y^{14}z^4.
\end{align*}

\item [(2)]
If $T\in G\{4, 7, 14, 18\}$, then
\begin{align*}
J_{C,T}&(w,z,x,y)=w^4x^{14} + 45w^4x^8y^6 + 42w^4x^6y^8 + 96w^4x^5y^9 + 303w^4x^4y^{10}\\
 &+ 312w^4x^3y^{11} + 129w^4x^2y^{12} + 72w^4xy^{13} + 24w^4y^{14} + 144w^3x^9y^5z\\
& + 120w^3x^7y^7z + 888w^3x^6y^8z + 2208w^3x^5y^9z + 
    2784w^3x^4y^{10}z \\
&+ 2616w^3x^3y^{11}z + 2232w^3x^2y^{12}z + 1056w^3xy^{13}z + 240w^3y^{14}z \\
&+
    90w^2x^{10}y^4z^2 + 198w^2x^8y^6z^2 + 1368w^2x^7y^7z^2 + 5166w^2x^6y^8z^2\\
 &+ 7920w^2x^5y^9z^2 +
    11718w^2x^4y^{10}z^2 + 13032w^2x^3y^{11}z^2 + 10080w^2x^2y^{12}z^2\\
& + 4752w^2xy^{13}z^2 + 972w^2y^{14}z^2 +
    24wx^{11}y^3z^3 + 96wx^9y^5z^3 + 744wx^8y^6z^3 \\
&+ 3864wx^7y^7z^3 + 8832wx^6y^8z^3 +
    14304wx^5y^9z^3 + 24072wx^4y^{10}z^3 \\
&+ 27168wx^3y^{11}z^3+ 20784wx^2y^{12}z^3 + 9072wxy^{13}z^3 + 
    1632wy^{14}z^3 \\
&+ 3x^{12}y^2z^4+ 3x^{10}y^4z^4 + 168x^9y^5z^4 + 852x^8y^6z^4 + 2184x^7y^7z^4\\
 &+
    5811x^6y^8z^4 + 12000x^5y^9z^4 + 18588x^4y^{10}z^4 + 20736x^3y^{11}z^4\\ 
&+ 14742x^2y^{12}z^4+   5712xy^{13}z^4 + 2145y^{14}z^4.
\end{align*}
\item [(3)]
If $T\in G\{2, 3, 5, 10 \}$, then
\begin{align*}
J_{C,T}&(w,z,x,y)=
w^4x^{14} + 51w^4x^8y^6 + 30w^4x^6y^8 + 150w^4x^5y^9 + 285w^4x^4y^{10} \\
&+ 228w^4x^3y^{11} + 153w^4x^2y^{12}+102w^4xy^{13} + 24w^4y^{14} 
+ 126w^3x^9y^5z \\
&+ 150w^3x^7y^7z + 744w^3x^6y^8z + 2262w^3x^5y^9z+2868w^3x^4y^{10}z \\
&+ 2778w^3x^3y^{11}z + 2112w^3x^2y^{12}z+ 1020w^3xy^{13}z + 228w^3y^{14}z \\
&+108w^2x^{10}y^4z^2 + 180w^2x^8y^6z^2+ 1476w^2x^7y^7z^2 + 5112w^2x^6y^8z^2 \\
&+ 8172w^2x^5y^9z^2 +
    11088w^2x^4y^{10}z^2 + 13212w^2x^3y^{11}z^2 + 10188w^2x^2y^{12}z^2 \\
&+ 4788w^2xy^{13}z^2 + 972w^2y^{14}z^2+
    18wx^{11}y^3z^3 + 90wx^9y^5z^3 \\
&+ 744wx^8y^6z^3 + 3882wx^7y^7z^3 + 8412wx^6y^8z^3 +
    14982wx^5y^9z^3 \\
&+ 23952wx^4y^{10}z^3 + 27060wx^3y^{11}z^3 + 20748wx^2y^{12}z^3 + 9072wxy^{13}z^3\\
&+ 1632wy^{14}z^3 + 3x^{12}y^2z^4 + 9x^{10}y^4z^4 + 150x^9y^5z^4 + 852x^8y^6z^4 \\
&+ 2352x^7y^7z^4+5577x^6y^8z^4 + 12030x^5y^9z^4 + 18624x^4y^{10}z^4 \\
&+ 20748x^3y^{11}z^4 + 14742x^2y^{12}z^4 +    5712xy^{13}z^4 + 2145y^{14}z^4.
\end{align*}
\item [(4)]
If $T\in G\{3, 6, 7, 17 \}$, then
\begin{align*}
J_{C,T}&(w,z,x,y)=w^4x^{14} + 48w^4x^8y^6 + 33w^4x^6y^8 + 120w^4x^5y^9 + 276w^4x^4y^{10} \\
&+ 288w^4x^3y^{11} + 174w^4x^2y^{12}+72w^4xy^{13} + 12w^4y^{14} + 138w^3x^9y^5z \\
&+ 138w^3x^7y^7z + 864w^3x^6y^8z + 2298w^3x^5y^9z + 
    2628w^3x^4y^{10}z \\
&+ 2694w^3x^3y^{11}z + 2232w^3x^2y^{12}z + 1068w^3xy^{13}z + 228w^3y^{14}z \\
&+ 90w^2x^{10}y^4z^2 + 198w^2x^8y^6z^2 + 1296w^2x^7y^7z^2 + 5058w^2x^6y^8z^2 \\
&+ 8532w^2x^5y^9z^2 + 11214w^2x^4y^{10}z^2 + 13032w^2x^3y^{11}z^2 \\
&+ 10116w^2x^2y^{12}z^2 + 4788w^2xy^{13}z^2 + 972w^2y^{14}z^2 
+ 30wx^{11}y^3z^3 \\
&+ 78wx^9y^5z^3 + 864wx^8y^6z^3 + 3918wx^7y^7z^3 + 8172wx^6y^8z^3 \\
&+14898wx^5y^9z^3 + 24072wx^4y^{10}z^3 + 27108wx^3y^{11}z^3 \\
&+ 20748wx^2y^{12}z^3 + 9072wxy^{13}z^3+ 
    1632wy^{14}z^3 + 12x^{10}y^4z^4 \\
&+ 120x^9y^5z^4 + 843x^8y^6z^4 + 2412x^7y^7z^4 + 5598x^6y^8z^4 \\
&+
    12000x^5y^9z^4 + 18612x^4y^{10}z^4 + 20748x^3y^{11}z^4 + 14742x^2y^{12}z^4 
\\ 
&+ 5712xy^{13}z^4 + 2145y^{14}z^4.
\end{align*}

\end{enumerate}
\end{thm}
\begin{proof}
Let $G:=\Aut(C)$. Then 
$G$ acts on $\widetilde{X}:=\binom{X}{4}$. 
By {\sc Magma}, we have 
\[
G\backslash \widetilde{X}=\{\{9, 12, 14,17\},\{4, 7, 14, 18\},\{2, 3, 5,10\},\{3,6, 7, 17\}\}. 
\]
Let 
\begin{align*}
&T_1:=\{9, 12, 14,17\}, T_2:=\{4, 7, 14, 18\}, \\
&T_3:=\{2, 3, 5,10\}, \mbox{ and } T_4:=\{3,6, 7, 17\}. 
\end{align*}
It is sufficient to determine $J_{C,T_i}\ (i\in \{1,\ldots,4\})$, 
and we perform brute force enumeration based on the definition 
by using {\sc Magma}. 

\end{proof}
\subsection{Harmonic weight enumerators of extended quaternary residue code of length 18}

In this subsection, 
we give harmonic weight enumerators of $C$.

\begin{thm}\label{thm:harm2}
Let $C$ be the extended quaternary residue code of length $18$ and 
$f$ be a harmonic function of degree $3$, which is 
an invariant of $\Aut(C)$. 
Then there exists $c_f$ such that 
\begin{align*}
w_{C,f}(x,y)=&c_f\times(-1224x^{12}y^{6} + 3672x^{10}y^{8} - 14688x^{9}y^{9} + 3672x^{8}y^{10} \\
		   &~~+ 51408x^{7}y^{11} - 52632x^{6}y^{12} + 7344x^{4}y^{14} + 2448x^{3}y^{15}). 
\end{align*}
\end{thm}
\begin{proof}
Let $G:=\Aut(C)$. By {\sc Magma}, 
$\Harm_3^{G}$ is a one-dimensional space, and 
let $\Harm_3^{G}=\langle f\rangle$. 
The process of calculation is as follows. 
Firstly, we compute the $G$-orbit of $\binom{X}{3}$: 
\[
\binom{X}{3}=GT_1 \sqcup G_{T_2}. 
\]
For $T=\{i_1,\ldots,i_t\}\in \binom{X}{t}$, 
$x_T:=x_{i_1}\cdots x_{i_t}$. 
For $i\in \{1,2\}$, 
let 
\[
g_i:=
\frac{1}{|G|}\sum_{g\in G}gx_{T_i}. 
\]
Secondly, we compute $a_1$ and $a_2$, which satisfy 
\[
\gamma(a_1g_1+ a_2g_2) = 0.
\]
Note that 
$f$ is listed online by Ishikawa \cite{Ishikawa}. 
By {\sc Magma}, we obtain the results. 
\end{proof}

\begin{proof}[Proof of Theorem \ref{thm:main2}]
Firstly, we show that $C_{13}$ is a $3$-design. 
We give two proofs. 
\begin{enumerate}
\item [(1)]
By Theorem \ref{thm:Jacobi2}, 
the coefficients of $z^{3} x^{18-13} y^{13-3}$ in 
$J_{C,T_i}$ $(i\in\{1,2\})$ are the same. 
Hence, $C_{13}$ is a $3$-design. 
\item [(2)]
By Theorem \ref{thm:harm2}, 
the coefficient of $x^{5} y^{13}$ in 
$w_{C,f}$ is zero. 
Because of Theorem \ref{thm:design-aut}, $C_{13}$ is a $3$-design. 
\end{enumerate}

Secondly, we show that for $\ell\neq 13$, 
$C_{\ell}$ is not a $3$-design 
whenever $C_{\ell}$ is non-empty. 
We give two proofs. 
\begin{enumerate}
\item[(1)]
By Theorem \ref{thm:Jacobi2}, 
for $\ell\neq 13$, 
the coefficients of $z^{3} x^{18-\ell} y^{\ell-3}$ in 
$J_{C,T_i}$ $(i\in\{1,\ldots,8\})$ are not the same. 
Hence, $C_{\ell}$ is not a $3$-design. 
\item[(2)]
By Theorem \ref{thm:harm2}, 
the coefficient of $x^{18-\ell} y^{\ell}$ in 
$w_{C,f}$ is non-zero. 
Because of Theorem \ref{thm:design-aut}, 
$C_{\ell}$ is not a $3$-design. 
\end{enumerate}

Thirdly, we show that $C_{13}$ is not a $4$-design. 
By Theorem \ref{thm:Jacobi3}, 
the coefficients of $z^{4} x^{18-13} y^{13-4}$ in 
$J_{C,T_i}$ $(i\in\{1,\ldots,4\})$ are not the same. 
Hence, $C_{13}$ is not a $4$-design. 
\end{proof}

\section{Concluding remarks}\label{sec:rem}

\begin{rem}
We have checked numerically that the cases $\delta(C)<s(C)$ do not exist in the extended ternary and quaternary quadratic residue codes 
with lengths of up to $20$ except for the cases in Theorem \ref{thm:main1} 
and \ref{thm:main2} and Remark \ref{rem:except}. 
\end{rem}

\begin{rem}\label{rem:final}

\begin{enumerate}
\item [(1)]
Let $C$ be the extended ternary quadratic residue code of length 14. Then by Theorem \ref{thm:main1}, $C_{10}$ is a $3$-design, and the weight distribution is 
\[
\{0,6,7,8,9,10,11,12,14\}. 
\]
Hence, the Assmus--Mattson theorem is inapplicable because $C_{10}$ is a $3$-design. However, by {\sc Magma}, $\Aut(C_{10})$ is 3-homogeneous, and so the transitivity argument is applicable in this case. 

\item [(2)]
Let $C$ be the extended quaternary quadratic residue code of length 18. Then by Theorem \ref{thm:main1}, $C_{13}$ is a $3$-design, and the weight distribution is 
\[
\{0,6,8,9,10,11,12,13,14,15,16,17,18\}. 
\]
Hence, the Assmus--Mattson theorem is inapplicable because $C_{13}$ is a $3$-design. 
However, we have not been able 
to determine the transitivity of $\Aut(C_{13})$.

\item [(3)]
Let $C$ be the extended quaternary quadratic residue code of length 18. Then by Remark \ref{rem:except}, $C_{10}$ is a $3$-design, and for the same reason as in Remark \ref{rem:final} (2), the Assmus--Mattson theorem is inapplicable because $C_{10}$ is a $3$-design. Furthermore, by {\sc Magma}, $\Aut(C_{10})$ is not 3-homogeneous, and so the transitivity argument is inapplicable in this case.

\end{enumerate}

\end{rem}

\section*{Acknowledgments}
The author would like to express gratitude to the anonymous reviewers for their
beneficial comments on an earlier version of the manuscript.
Additionally, the author extends thanks to Professors Tsuyoshi Miezaki and Akihiro Munemasa 
for their helpful discussions and comments. 


\end{document}